\documentclass[preprint,12pt]{elsarticle}

\usepackage{amscd}
\usepackage{amsmath}
\usepackage{amsthm}
\usepackage{amsfonts}
\usepackage{graphicx}
\usepackage{amssymb}
\usepackage{color}
\usepackage{amsbsy}
\usepackage{graphicx}
\usepackage{amssymb,color,amsbsy}

\usepackage{float}
\usepackage{comment}

\usepackage[ruled]{algorithm2e}

\usepackage[matrix,arrow]{xy}
\usepackage{mathabx}   

\DeclareMathAlphabet{\mathpzc}{OT1}{pzc}{m}{it}

\usepackage{pgfpages}
\usepackage{tikz}
\usetikzlibrary{shapes.geometric}
\usetikzlibrary{decorations.pathmorphing}
\usetikzlibrary{arrows}
\usetikzlibrary{backgrounds}
\usetikzlibrary{positioning}
\usetikzlibrary{fit}
\usepackage{caption}
\usepackage{amsthm}

\usepackage{amsmath}

\usepackage[pagebackref=true, bookmarksopen=true,colorlinks=true, linkcolor=red,citecolor=blue]{hyperref}

\usepackage[capitalise]{cleveref}  

\global\long\def\ii{\cap}%

\global\long\def\u{\cup}%

\global\long\def\I{\bigcap}%

\global\long\def\U{\bigcup}%

\global\long\def\s{\subset}%

\global\long\def\P{\prime}%

\global\long\def\ñ{\sim}%

\usepackage{tikz-cd}

\usepackage{tkz-graph}

\usepackage{multicol}

\usepackage{subcaption}

\newtheorem{theorem}{Theorem}[section]

\newtheorem{conjecture}[theorem]{Conjecture}
\newtheorem{corollary}[theorem]{Corollary}

\newtheorem{lemma}[theorem]{Lemma}

\newtheorem{problem}[theorem]{Problem}

\newtheorem{remark}[theorem]{Remark}

\usepackage{comment}
\begin{document}

\begin{abstract}
Let $\alpha(G)$ denote the cardinality of a maximum independent set.
An independent set $I$ of $G$ is critical if $\left|I\right|-\left|N(I)\right|\ge\left|J\right|-\left|N(J)\right|$
for every independent set $J$ of $G$. Let $\text{core}(G)$
and $\text{corona}(G)$ be the intersection/union of all
maximum independent sets of $G$. Let
$\text{ker}(G)$ and $\text{diadem}(G)$ be the intersection/union of all
critical independent sets of $G$. In this paper we prove
that 
\[
\left|\text{corona}(G)\right|+\left|\text{core}(G)\right|\le2\alpha(G)+k,
\]
\noindent where $k$ is the number of vertex-distinct odd cycles
in $G$, thus confirming a recent conjecture in the area. Moreover,
we prove that 
\[
\left|\text{nucleus}(G)\right|+\left|\text{diadem}(G)\right|\le2\alpha(G),
\]
\noindent thereby confirming another conjecture (Levit--Mandrescu 2014). As
an application of these facts, we obtain a chain of inequalities 
\[
\left|\text{nucleus}(G)\right|+\left|\text{diadem}(G)\right|\le2\alpha(G)\le\left|\text{corona}(G)\right|+\left|\text{core}(G)\right|\le2\alpha(G)+k.
\]
\noindent The paper concludes with a collection of related open problems.
\end{abstract}

\begin{frontmatter} 
	\title{Inequalities Involving Core, Corona, and Critical Sets in General Graphs}

	\cortext[cor1]{Corresponding author}
	
	\author[IMASL,DEPTO]{Adrián Pastine} 	\ead{agpastine@unsl.edu.ar} 
	\author[IMASL,DEPTO]{Kevin Pereyra\corref{cor1}} 	\ead{kdpereyra@unsl.edu.ar}


	\address[IMASL]{Instituto de Matem\'atica Aplicada San Luis, Universidad Nacional de San Luis and CONICET, San Luis, Argentina.}
	\address[DEPTO]{Departamento de Matem\'atica, Universidad Nacional de San Luis, San Luis, Argentina.} 	

	\date{Received: date / Accepted: date} 
	
\end{frontmatter} %

\section{Introduction}\label{introasdm}

Let $\alpha(G)$ denote the cardinality of a maximum independent set,
and let $\mu(G)$ be the size of a maximum matching in $G=(V,E)$.
It is known that $\alpha(G)+\mu(G)$ equals the order of $G$,
in which case $G$ is a König--Egerváry graph 
\cite{deming1979independence,gavril1977testing,stersoul1979characterization}.
Various properties of König--Egerváry graphs were presented in 
\cite{bourjolly2009node,jarden2017two,levit2006alpha,levit2012critical}.
It is known that every bipartite graph is a König--Egerváry graph 
\cite{egervary1931combinatorial}. 

Let $\Omega^{*}(G)=\left\{ S:S\textnormal{ is an independent set of }G\right\}$,
$\Omega(G)=\{S:S$ is a maximum independent set of $G\}$,
$\textnormal{core}(G)=\I\left\{ S:S\in\Omega(G)\right\}$ 
\cite{levit2003alpha+}, and 
$\textnormal{corona}(G)=\U\left\{ S:S\in\Omega(G)\right\}$ 
\cite{boros2002number}. 
The number $d_{G}(X)=\left|X\right|-\left|N(X)\right|$ is the 
difference of the set $X\s V(G)$, and 
$d(G)=\max\{d_{G}(X):X\s V(G)\}$ is called the \emph{critical difference} of $G$.
A set $U\s V(G)$ is \emph{critical} if $d_{G}(U)=d(G)$ 
\cite{zhang1990finding}. 
The number $d_{I}(G)=\max\left\{ d_{G}(X):X\in\Omega^{*}(G)\right\}$ 
is called the \emph{critical independence difference} of $G$. 
If a set $X\s\Omega^{*}(G)$ satisfies $d_{G}(X)=d_{I}(G)$, then it is called 
a \emph{critical independent set} \cite{zhang1990finding}. 
Clearly, $d(G)\ge d_{I}(G)$ holds for every graph. 
It is known that $d(G)=d_{I}(G)$ for all graphs 
\cite{zhang1990finding}. 
We define $\textnormal{ker}(G)=\I\{ S:S$ is a critical independent set of $G\}$ 
\cite{levit2012vertices,lorentzen1966notes,schrijver2003combinatorial}.
Let $\textnormal{nucleus}(G)=\I\{S:S$ is a maximum critical independent 
set of $G\}$ \cite{jarden2019monotonic}, and 
$\textnormal{diadem}(G)=\U\{S:S$ is a critical independent set of $G\}$ 
\cite{short2015some}.

A graph $G$ is almost bipartite if it has exactly one odd cycle. 
It is known that $\textnormal{ker}(G)\s\textnormal{core}(G)$ holds for every graph 
\cite{levit2012vertices}. 
Equality holds for bipartite graphs \cite{levit2013critical}, 
for unicyclic non–König–Egerváry graphs \cite{levit2011core}, 
for almost bipartite non–König–Egerváry graphs 
\cite{levit2025almost,levit2024almost}, and 
for $R$-disjoint graphs, which generalize the class of almost bipartite 
non–König–Egerváry graphs \cite{kevin2025RDG}.

An almost bipartite non-König-Egerváry graph satisfies $\left|\text{corona}(G)\right|+\left|\text{core}(G)\right|=2\alpha(G)+1$
\cite{levit2025almost}. An $R$-disjoint graph with exactly $k$ odd cycles
satisfies
\[
\left|\text{corona}(G)\right|+\left|\text{ker}(G)\right|=2\alpha(G)+k,
\]
\noindent see \cite{kevin2025RDG}. However, it is easy to see that this equality
does not hold for every graph (see \cref{Figura3}).

\begin{figure}[h!]
	
	\begin{center}

		\tikzset{every picture/.style={line width=0.75pt}} 
		
		\begin{tikzpicture}[x=0.75pt,y=0.75pt,yscale=-1,xscale=1]
			
			\draw    (198.2,137.55) -- (165.85,98.47) ;
			\draw [shift={(165.85,98.47)}, rotate = 230.38] [color={rgb, 255:red, 0; green, 0; blue, 0 }  ][fill={rgb, 255:red, 0; green, 0; blue, 0 }  ][line width=0.75]      (0, 0) circle [x radius= 3.35, y radius= 3.35]   ;
			\draw [shift={(198.2,137.55)}, rotate = 230.38] [color={rgb, 255:red, 0; green, 0; blue, 0 }  ][fill={rgb, 255:red, 0; green, 0; blue, 0 }  ][line width=0.75]      (0, 0) circle [x radius= 3.35, y radius= 3.35]   ;
			\draw    (210.83,91.8) -- (198.2,137.55) ;
			\draw [shift={(198.2,137.55)}, rotate = 105.43] [color={rgb, 255:red, 0; green, 0; blue, 0 }  ][fill={rgb, 255:red, 0; green, 0; blue, 0 }  ][line width=0.75]      (0, 0) circle [x radius= 3.35, y radius= 3.35]   ;
			\draw [shift={(210.83,91.8)}, rotate = 105.43] [color={rgb, 255:red, 0; green, 0; blue, 0 }  ][fill={rgb, 255:red, 0; green, 0; blue, 0 }  ][line width=0.75]      (0, 0) circle [x radius= 3.35, y radius= 3.35]   ;
			\draw    (223.53,165.25) -- (198.2,137.55) ;
			\draw [shift={(198.2,137.55)}, rotate = 227.55] [color={rgb, 255:red, 0; green, 0; blue, 0 }  ][fill={rgb, 255:red, 0; green, 0; blue, 0 }  ][line width=0.75]      (0, 0) circle [x radius= 3.35, y radius= 3.35]   ;
			\draw [shift={(223.53,165.25)}, rotate = 227.55] [color={rgb, 255:red, 0; green, 0; blue, 0 }  ][fill={rgb, 255:red, 0; green, 0; blue, 0 }  ][line width=0.75]      (0, 0) circle [x radius= 3.35, y radius= 3.35]   ;
			\draw    (265.1,172.58) -- (223.53,165.25) ;
			\draw [shift={(223.53,165.25)}, rotate = 190.01] [color={rgb, 255:red, 0; green, 0; blue, 0 }  ][fill={rgb, 255:red, 0; green, 0; blue, 0 }  ][line width=0.75]      (0, 0) circle [x radius= 3.35, y radius= 3.35]   ;
			\draw [shift={(265.1,172.58)}, rotate = 190.01] [color={rgb, 255:red, 0; green, 0; blue, 0 }  ][fill={rgb, 255:red, 0; green, 0; blue, 0 }  ][line width=0.75]      (0, 0) circle [x radius= 3.35, y radius= 3.35]   ;
			\draw    (312.39,170.25) -- (265.1,172.58) ;
			\draw [shift={(265.1,172.58)}, rotate = 177.17] [color={rgb, 255:red, 0; green, 0; blue, 0 }  ][fill={rgb, 255:red, 0; green, 0; blue, 0 }  ][line width=0.75]      (0, 0) circle [x radius= 3.35, y radius= 3.35]   ;
			\draw [shift={(312.39,170.25)}, rotate = 177.17] [color={rgb, 255:red, 0; green, 0; blue, 0 }  ][fill={rgb, 255:red, 0; green, 0; blue, 0 }  ][line width=0.75]      (0, 0) circle [x radius= 3.35, y radius= 3.35]   ;
			\draw    (349.17,141.31) -- (312.39,170.25) ;
			\draw [shift={(312.39,170.25)}, rotate = 141.8] [color={rgb, 255:red, 0; green, 0; blue, 0 }  ][fill={rgb, 255:red, 0; green, 0; blue, 0 }  ][line width=0.75]      (0, 0) circle [x radius= 3.35, y radius= 3.35]   ;
			\draw [shift={(349.17,141.31)}, rotate = 141.8] [color={rgb, 255:red, 0; green, 0; blue, 0 }  ][fill={rgb, 255:red, 0; green, 0; blue, 0 }  ][line width=0.75]      (0, 0) circle [x radius= 3.35, y radius= 3.35]   ;
			\draw    (351.6,193.35) -- (312.39,170.25) ;
			\draw [shift={(312.39,170.25)}, rotate = 210.51] [color={rgb, 255:red, 0; green, 0; blue, 0 }  ][fill={rgb, 255:red, 0; green, 0; blue, 0 }  ][line width=0.75]      (0, 0) circle [x radius= 3.35, y radius= 3.35]   ;
			\draw [shift={(351.6,193.35)}, rotate = 210.51] [color={rgb, 255:red, 0; green, 0; blue, 0 }  ][fill={rgb, 255:red, 0; green, 0; blue, 0 }  ][line width=0.75]      (0, 0) circle [x radius= 3.35, y radius= 3.35]   ;
			\draw    (349.17,141.31) -- (351.6,193.35) ;
			\draw [shift={(351.6,193.35)}, rotate = 87.33] [color={rgb, 255:red, 0; green, 0; blue, 0 }  ][fill={rgb, 255:red, 0; green, 0; blue, 0 }  ][line width=0.75]      (0, 0) circle [x radius= 3.35, y radius= 3.35]   ;
			\draw [shift={(349.17,141.31)}, rotate = 87.33] [color={rgb, 255:red, 0; green, 0; blue, 0 }  ][fill={rgb, 255:red, 0; green, 0; blue, 0 }  ][line width=0.75]      (0, 0) circle [x radius= 3.35, y radius= 3.35]   ;
			\draw    (249.4,123.28) -- (265.1,172.58) ;
			\draw [shift={(265.1,172.58)}, rotate = 72.34] [color={rgb, 255:red, 0; green, 0; blue, 0 }  ][fill={rgb, 255:red, 0; green, 0; blue, 0 }  ][line width=0.75]      (0, 0) circle [x radius= 3.35, y radius= 3.35]   ;
			\draw [shift={(249.4,123.28)}, rotate = 72.34] [color={rgb, 255:red, 0; green, 0; blue, 0 }  ][fill={rgb, 255:red, 0; green, 0; blue, 0 }  ][line width=0.75]      (0, 0) circle [x radius= 3.35, y radius= 3.35]   ;
			\draw    (284.36,121.56) -- (265.1,172.58) ;
			\draw [shift={(265.1,172.58)}, rotate = 110.68] [color={rgb, 255:red, 0; green, 0; blue, 0 }  ][fill={rgb, 255:red, 0; green, 0; blue, 0 }  ][line width=0.75]      (0, 0) circle [x radius= 3.35, y radius= 3.35]   ;
			\draw [shift={(284.36,121.56)}, rotate = 110.68] [color={rgb, 255:red, 0; green, 0; blue, 0 }  ][fill={rgb, 255:red, 0; green, 0; blue, 0 }  ][line width=0.75]      (0, 0) circle [x radius= 3.35, y radius= 3.35]   ;
			\draw    (284.36,121.56) -- (249.4,123.28) ;
			\draw [shift={(249.4,123.28)}, rotate = 177.17] [color={rgb, 255:red, 0; green, 0; blue, 0 }  ][fill={rgb, 255:red, 0; green, 0; blue, 0 }  ][line width=0.75]      (0, 0) circle [x radius= 3.35, y radius= 3.35]   ;
			\draw [shift={(284.36,121.56)}, rotate = 177.17] [color={rgb, 255:red, 0; green, 0; blue, 0 }  ][fill={rgb, 255:red, 0; green, 0; blue, 0 }  ][line width=0.75]      (0, 0) circle [x radius= 3.35, y radius= 3.35]   ;
			\draw [color={rgb, 255:red, 255; green, 0; blue, 0 }  ,draw opacity=1 ][line width=1.5]    (349.17,141.31) -- (351.6,193.35) ;
			\draw [color={rgb, 255:red, 255; green, 0; blue, 0 }  ,draw opacity=1 ][line width=1.5]    (312.39,170.25) -- (265.1,172.58) ;
			\draw [color={rgb, 255:red, 255; green, 0; blue, 0 }  ,draw opacity=1 ][line width=1.5]    (284.36,121.56) -- (249.4,123.28) ;
			\draw [color={rgb, 255:red, 255; green, 0; blue, 0 }  ,draw opacity=1 ][line width=1.5]    (223.53,165.25) -- (198.2,137.55) ;
			\draw    (198.2,137.55) ;
			\draw [shift={(198.2,137.55)}, rotate = 0] [color={rgb, 255:red, 0; green, 0; blue, 0 }  ][fill={rgb, 255:red, 0; green, 0; blue, 0 }  ][line width=0.75]      (0, 0) circle [x radius= 3.35, y radius= 3.35]   ;
			\draw [shift={(198.2,137.55)}, rotate = 0] [color={rgb, 255:red, 0; green, 0; blue, 0 }  ][fill={rgb, 255:red, 0; green, 0; blue, 0 }  ][line width=0.75]      (0, 0) circle [x radius= 3.35, y radius= 3.35]   ;
			\draw    (223.53,165.25) ;
			\draw [shift={(223.53,165.25)}, rotate = 0] [color={rgb, 255:red, 0; green, 0; blue, 0 }  ][fill={rgb, 255:red, 0; green, 0; blue, 0 }  ][line width=0.75]      (0, 0) circle [x radius= 3.35, y radius= 3.35]   ;
			\draw [shift={(223.53,165.25)}, rotate = 0] [color={rgb, 255:red, 0; green, 0; blue, 0 }  ][fill={rgb, 255:red, 0; green, 0; blue, 0 }  ][line width=0.75]      (0, 0) circle [x radius= 3.35, y radius= 3.35]   ;
			\draw    (349.17,141.31) ;
			\draw [shift={(349.17,141.31)}, rotate = 0] [color={rgb, 255:red, 0; green, 0; blue, 0 }  ][fill={rgb, 255:red, 0; green, 0; blue, 0 }  ][line width=0.75]      (0, 0) circle [x radius= 3.35, y radius= 3.35]   ;
			\draw [shift={(349.17,141.31)}, rotate = 0] [color={rgb, 255:red, 0; green, 0; blue, 0 }  ][fill={rgb, 255:red, 0; green, 0; blue, 0 }  ][line width=0.75]      (0, 0) circle [x radius= 3.35, y radius= 3.35]   ;
			\draw    (312.39,170.25) ;
			\draw [shift={(312.39,170.25)}, rotate = 0] [color={rgb, 255:red, 0; green, 0; blue, 0 }  ][fill={rgb, 255:red, 0; green, 0; blue, 0 }  ][line width=0.75]      (0, 0) circle [x radius= 3.35, y radius= 3.35]   ;
			\draw [shift={(312.39,170.25)}, rotate = 0] [color={rgb, 255:red, 0; green, 0; blue, 0 }  ][fill={rgb, 255:red, 0; green, 0; blue, 0 }  ][line width=0.75]      (0, 0) circle [x radius= 3.35, y radius= 3.35]   ;
			\draw    (284.36,121.56) ;
			\draw [shift={(284.36,121.56)}, rotate = 0] [color={rgb, 255:red, 0; green, 0; blue, 0 }  ][fill={rgb, 255:red, 0; green, 0; blue, 0 }  ][line width=0.75]      (0, 0) circle [x radius= 3.35, y radius= 3.35]   ;
			\draw [shift={(284.36,121.56)}, rotate = 0] [color={rgb, 255:red, 0; green, 0; blue, 0 }  ][fill={rgb, 255:red, 0; green, 0; blue, 0 }  ][line width=0.75]      (0, 0) circle [x radius= 3.35, y radius= 3.35]   ;
			\draw    (265.1,172.58) ;
			\draw [shift={(265.1,172.58)}, rotate = 0] [color={rgb, 255:red, 0; green, 0; blue, 0 }  ][fill={rgb, 255:red, 0; green, 0; blue, 0 }  ][line width=0.75]      (0, 0) circle [x radius= 3.35, y radius= 3.35]   ;
			\draw [shift={(265.1,172.58)}, rotate = 0] [color={rgb, 255:red, 0; green, 0; blue, 0 }  ][fill={rgb, 255:red, 0; green, 0; blue, 0 }  ][line width=0.75]      (0, 0) circle [x radius= 3.35, y radius= 3.35]   ;
			\draw    (249.4,123.28) ;
			\draw [shift={(249.4,123.28)}, rotate = 0] [color={rgb, 255:red, 0; green, 0; blue, 0 }  ][fill={rgb, 255:red, 0; green, 0; blue, 0 }  ][line width=0.75]      (0, 0) circle [x radius= 3.35, y radius= 3.35]   ;
			\draw [shift={(249.4,123.28)}, rotate = 0] [color={rgb, 255:red, 0; green, 0; blue, 0 }  ][fill={rgb, 255:red, 0; green, 0; blue, 0 }  ][line width=0.75]      (0, 0) circle [x radius= 3.35, y radius= 3.35]   ;
			\draw    (351.6,193.35) ;
			\draw [shift={(351.6,193.35)}, rotate = 0] [color={rgb, 255:red, 0; green, 0; blue, 0 }  ][fill={rgb, 255:red, 0; green, 0; blue, 0 }  ][line width=0.75]      (0, 0) circle [x radius= 3.35, y radius= 3.35]   ;
			\draw [shift={(351.6,193.35)}, rotate = 0] [color={rgb, 255:red, 0; green, 0; blue, 0 }  ][fill={rgb, 255:red, 0; green, 0; blue, 0 }  ][line width=0.75]      (0, 0) circle [x radius= 3.35, y radius= 3.35]   ;
			\draw  [fill={rgb, 255:red, 74; green, 144; blue, 226 }  ,fill opacity=0.1 ][dash pattern={on 0.84pt off 2.51pt}] (151.33,100.79) .. controls (149.82,89.85) and (164.11,78.84) .. (183.26,76.19) .. controls (202.41,73.54) and (219.16,80.27) .. (220.67,91.21) .. controls (222.18,102.15) and (207.89,113.16) .. (188.74,115.81) .. controls (169.59,118.46) and (152.84,111.73) .. (151.33,100.79) -- cycle ;
			\draw [color={rgb, 255:red, 0; green, 0; blue, 255 }  ,draw opacity=1 ]   (249.4,123.28) ;
			\draw [shift={(249.4,123.28)}, rotate = 0] [color={rgb, 255:red, 0; green, 0; blue, 255 }  ,draw opacity=1 ][fill={rgb, 255:red, 0; green, 0; blue, 255 }  ,fill opacity=1 ][line width=0.75]      (0, 0) circle [x radius= 3.69, y radius= 3.69]   ;
			\draw [shift={(249.4,123.28)}, rotate = 0] [color={rgb, 255:red, 0; green, 0; blue, 255 }  ,draw opacity=1 ][fill={rgb, 255:red, 0; green, 0; blue, 255 }  ,fill opacity=1 ][line width=0.75]      (0, 0) circle [x radius= 3.69, y radius= 3.69]   ;
			\draw [color={rgb, 255:red, 0; green, 0; blue, 255 }  ,draw opacity=1 ]   (165.85,98.47) ;
			\draw [shift={(165.85,98.47)}, rotate = 0] [color={rgb, 255:red, 0; green, 0; blue, 255 }  ,draw opacity=1 ][fill={rgb, 255:red, 0; green, 0; blue, 255 }  ,fill opacity=1 ][line width=0.75]      (0, 0) circle [x radius= 3.69, y radius= 3.69]   ;
			\draw [shift={(165.85,98.47)}, rotate = 0] [color={rgb, 255:red, 0; green, 0; blue, 255 }  ,draw opacity=1 ][fill={rgb, 255:red, 0; green, 0; blue, 255 }  ,fill opacity=1 ][line width=0.75]      (0, 0) circle [x radius= 3.69, y radius= 3.69]   ;
			\draw [color={rgb, 255:red, 0; green, 0; blue, 255 }  ,draw opacity=1 ]   (210.83,91.8) ;
			\draw [shift={(210.83,91.8)}, rotate = 0] [color={rgb, 255:red, 0; green, 0; blue, 255 }  ,draw opacity=1 ][fill={rgb, 255:red, 0; green, 0; blue, 255 }  ,fill opacity=1 ][line width=0.75]      (0, 0) circle [x radius= 3.69, y radius= 3.69]   ;
			\draw [shift={(210.83,91.8)}, rotate = 0] [color={rgb, 255:red, 0; green, 0; blue, 255 }  ,draw opacity=1 ][fill={rgb, 255:red, 0; green, 0; blue, 255 }  ,fill opacity=1 ][line width=0.75]      (0, 0) circle [x radius= 3.69, y radius= 3.69]   ;
			\draw [color={rgb, 255:red, 0; green, 0; blue, 255 }  ,draw opacity=1 ]   (223.53,165.25) ;
			\draw [shift={(223.53,165.25)}, rotate = 0] [color={rgb, 255:red, 0; green, 0; blue, 255 }  ,draw opacity=1 ][fill={rgb, 255:red, 0; green, 0; blue, 255 }  ,fill opacity=1 ][line width=0.75]      (0, 0) circle [x radius= 3.69, y radius= 3.69]   ;
			\draw [shift={(223.53,165.25)}, rotate = 0] [color={rgb, 255:red, 0; green, 0; blue, 255 }  ,draw opacity=1 ][fill={rgb, 255:red, 0; green, 0; blue, 255 }  ,fill opacity=1 ][line width=0.75]      (0, 0) circle [x radius= 3.69, y radius= 3.69]   ;
			\draw [color={rgb, 255:red, 0; green, 0; blue, 255 }  ,draw opacity=1 ]   (351.6,193.35) ;
			\draw [shift={(351.6,193.35)}, rotate = 0] [color={rgb, 255:red, 0; green, 0; blue, 255 }  ,draw opacity=1 ][fill={rgb, 255:red, 0; green, 0; blue, 255 }  ,fill opacity=1 ][line width=0.75]      (0, 0) circle [x radius= 3.69, y radius= 3.69]   ;
			\draw [shift={(351.6,193.35)}, rotate = 0] [color={rgb, 255:red, 0; green, 0; blue, 255 }  ,draw opacity=1 ][fill={rgb, 255:red, 0; green, 0; blue, 255 }  ,fill opacity=1 ][line width=0.75]      (0, 0) circle [x radius= 3.69, y radius= 3.69]   ;
			\draw [color={rgb, 255:red, 0; green, 0; blue, 255 }  ,draw opacity=1 ]   (312.39,170.25) ;
			\draw [shift={(312.39,170.25)}, rotate = 0] [color={rgb, 255:red, 0; green, 0; blue, 255 }  ,draw opacity=1 ][fill={rgb, 255:red, 0; green, 0; blue, 255 }  ,fill opacity=1 ][line width=0.75]      (0, 0) circle [x radius= 3.69, y radius= 3.69]   ;
			\draw [shift={(312.39,170.25)}, rotate = 0] [color={rgb, 255:red, 0; green, 0; blue, 255 }  ,draw opacity=1 ][fill={rgb, 255:red, 0; green, 0; blue, 255 }  ,fill opacity=1 ][line width=0.75]      (0, 0) circle [x radius= 3.69, y radius= 3.69]   ;
			\draw [color={rgb, 255:red, 0; green, 0; blue, 255 }  ,draw opacity=1 ]   (349.17,141.31) ;
			\draw [shift={(349.17,141.31)}, rotate = 0] [color={rgb, 255:red, 0; green, 0; blue, 255 }  ,draw opacity=1 ][fill={rgb, 255:red, 0; green, 0; blue, 255 }  ,fill opacity=1 ][line width=0.75]      (0, 0) circle [x radius= 3.69, y radius= 3.69]   ;
			\draw [shift={(349.17,141.31)}, rotate = 0] [color={rgb, 255:red, 0; green, 0; blue, 255 }  ,draw opacity=1 ][fill={rgb, 255:red, 0; green, 0; blue, 255 }  ,fill opacity=1 ][line width=0.75]      (0, 0) circle [x radius= 3.69, y radius= 3.69]   ;
			\draw [color={rgb, 255:red, 0; green, 0; blue, 255 }  ,draw opacity=1 ]   (284.36,121.56) ;
			\draw [shift={(284.36,121.56)}, rotate = 0] [color={rgb, 255:red, 0; green, 0; blue, 255 }  ,draw opacity=1 ][fill={rgb, 255:red, 0; green, 0; blue, 255 }  ,fill opacity=1 ][line width=0.75]      (0, 0) circle [x radius= 3.69, y radius= 3.69]   ;
			\draw [shift={(284.36,121.56)}, rotate = 0] [color={rgb, 255:red, 0; green, 0; blue, 255 }  ,draw opacity=1 ][fill={rgb, 255:red, 0; green, 0; blue, 255 }  ,fill opacity=1 ][line width=0.75]      (0, 0) circle [x radius= 3.69, y radius= 3.69]   ;

			\draw (163,53.4) node [anchor=north west][inner sep=0.75pt]    {$\text{ker}( G)$};

		\end{tikzpicture}

	\end{center}
	\caption{In this example, the vertices in $\text{corona}(G)$ are shown in blue.
		Note that $\left|\text{corona}(G)\right|+\left|\text{ker}(G)\right|=8+2=10<2\cdot5+2=2\alpha(G)+2,$ moreover
		$2$ is the number of odd cycles in $G$.
	}
	\label{Figura3}
	
\end{figure}
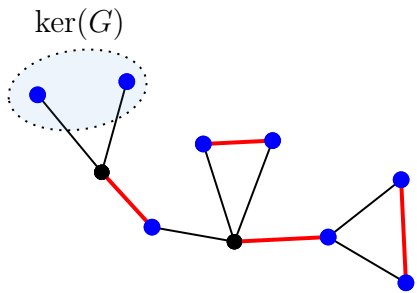

The previous results motivate the following conjecture.

\begin{conjecture}
	[\cite{kevinBAB}\label{ADSKMSkk}\label{babconj}] For every graph $G$, 
	\[
	\left|\textnormal{corona}(G)\right|+\left|\textnormal{ker}(G)\right|\le2\alpha(G)+k,
	\]
	\noindent where $k$ is the number of odd cycles in $G$. 
\end{conjecture}

In \cite{levit2014critical} it is shown that $\left|\text{corona}(G)\right|+\left|\text{core}(G)\right|=2\alpha(G)$
for every König-Egerváry graph, and in \cite{levit2014set} that every
graph satisfies $\left|\text{corona}(G)\right|+\left|\text{core}(G)\right|\ge2\alpha(G)$.
Moreover, the following conjecture is posed.

\begin{conjecture}[\label{asdasd1}\cite{levit2014critical}] For every graph $G$, 
	\[
	\left|\textnormal{ker}(G)\right|+\left|\textnormal{diadem}(G)\right|\le2\alpha(G).
	\]

\end{conjecture}

In this paper we show that $\left|\text{corona}(G)\right|+\left|\text{core}(G)\right|\le2\alpha(G)+k,$ where $k$ is the number of vertex distinct odd cycles. This is actually a strengthening of  \cref{ADSKMSkk}. Moreover, we show that \cref{asdasd1} holds in a stronger form,
 and as a consequence we obtain the following
inequalities.

\begin{theorem}
	For every graph $G$, 
	\begin{align*}
		\left|\textnormal{nucleus}(G)\right|+\left|\textnormal{diadem}(G)\right| & \le2\alpha(G)\\
		& \le\left|\textnormal{corona}(G)\right|+\left|\textnormal{core}(G)\right|\\
		& \le2\alpha(G)+k.
	\end{align*}
	\noindent where $k$ is the number of vertex distinct odd cycles in $G$.
\end{theorem} 

In \cref{sss1}, we provide a brief review of the notation that will be used throughout the paper. In \cref{introasdm}, we present the necessary context for the paper. In \cref{sss2} we prove the main results of the paper. We conclude the paper in \cref{sss5} by presenting some open problems.

\section{Preliminaries}\label{sss1}
All graphs considered in this paper are finite, undirected, and simple. 
For any undefined terminology or notation, we refer the reader to 
Lovász and Plummer \cite{LP} or Diestel \cite{Distel}.

Let \( G = (V, E) \) be a simple graph, where \( V = V(G) \) is the finite set of vertices and \( E = E(G) \) is the set of edges, with \( E \subseteq \{\{u, v\} : u, v \in V, u \neq v\} \). We denote the edge \( e=\{u, v\} \) as \( uv \). A subgraph of \( G \) is a graph \( H \) such that \( V(H) \subseteq V(G) \) and \( E(H) \subseteq E(G) \). A subgraph \( H \) of \( G \) is called a \textit{spanning} subgraph if \( V(H) = V(G) \). 

Let \( e \in E(G) \) and \( v \in V(G) \). We define \( G - e := (V, E \setminus \{e\}) \) and \( G - v := (V \setminus \{v\}, \{uw \in E : u,w \neq v\}) \). If \( X \subseteq V(G) \), the \textit{induced} subgraph of \( G \) by \( X \) is the subgraph \( G[X]=(X,F) \), where \( F:=\{uv \!\in\! E(G) : u, v \!\in \! X\} \).

Given a vertex set $S \subseteq V(G)$, we denote by $\partial(S)$ 
the set of edges having one endpoint in $S$ and the other in $V(G)-S$. 
We also denote by $\partial(S)$ the set of vertices in $S$ 
for which there exists an edge with one endpoint in $S$ and the other in $V(G)-S$.

A \textit{matching} \(M\) in a graph \(G\) is a set of pairwise non-adjacent edges. 
The \textit{matching number} of \(G\), denoted by  \(\mu(G)\), is the maximum cardinality of any matching in \(G\). 
Matchings induce an involution on the vertex set of the graph: \(M:V(G)\rightarrow V(G)\), where \(M(v)=u\) if \(uv \in M\), and \(M(v)=v\) otherwise. 
If \(S, U \subseteq V(G)\) with \(S \cap U = \emptyset\), we say that \(M\) is a matching from \(S\) to \(U\) if \(M(S) \subseteq U\).

A vertex set \( S \subseteq V \) is \textit{independent} if, for every pair of vertices \( u, v \in S \), we have \( uv \notin E \). 
The number of vertices in a maximum independent set is denoted by \( \alpha(G) \). 
A \textit{bipartite} graph is a graph whose vertex set can be partitioned into two disjoint independent sets. 
The number of vertices in a graph is called the \textit{order} of the graph.
A \textit{cycle} in $G$ is called \textit{odd} (resp. \textit{even}) if it has an odd (resp. even) number of edges.

\section{Main Results}\label{sss2}

Before proving the main result, we need some known results.

\begin{theorem}[\label{levit1}\cite{levit2012vertices}]
	For every graph $G$, we have $\textnormal{ker}(G)\subseteq\textnormal{core}(G)$.
\end{theorem}

Two cycles $C$ and $C^{\P}$ are considered \emph{vertex-distinct}
if $V(C)\neq V(C^{\P})$. Note that distinct cycles are not necessarily vertex-distinct.
A set of cycles is vertex-distinct if they are pairwise vertex-distinct.

\begin{theorem}[\cite{jarden2017two}\label{asijdnasn}]
	 For a graph G the following assertions are equivalent:
\begin{enumerate}
	\item $G$ is a König--Egerváry graph;
	\item for every $S_{1},S_{2}\in\Omega(G)$, there ir a matching from $V(G)-(S_{1}\u S_{2})$
	into $S_{1}\ii S_{2}$; 
	\item there exist $S_{1},S_{2}\in\Omega(G)$, such that there is a matching
	from $V(G)-(S_{1}\u S_{2})$ into $S_{1}\ii S_{2}$. 
\end{enumerate}
\end{theorem}

\begin{theorem}\label{mainlema1}
	For every graph $G$, 
	\[
	\left|\textnormal{corona}(G)\right|+\left|\textnormal{core}(G)\right|\le2\alpha(G)+k,
	\]
	\noindent where $k$ is the maximum cardinality of a set of
	vertex-distinct odd cycles of $G$. 
\end{theorem}

\begin{proof}
	Let $S_1,S_2$ be two maximum independent sets of $G$.
	Note that $G[S_1\cup S_2]$ is a bipartite graph. Let $X\subseteq\text{corona}(G)$
	be a maximum-cardinality set such that $G[S_1\cup S_2\cup X]$ is a bipartite
	graph. Then, by \cref{asijdnasn}, there exists a matching from $X$ into
	$S_1\cap S_2$. Hence $\left|X\right|\le\left|S_{1}\cap S_{2}\right|$.
	Therefore 
	\begin{eqnarray*}
		2\alpha(G) & = & \left|S_{1}\cup S_{2}\right|+2\left|S_{1}\cap S_{2}\right|\\
		& \ge & \left|S_{1}\cup S_{2}\right|+\left|X\right|+\left|S_{1}\cap S_{2}\right|\\
		& \ge & \left|S_{1}\cup S_{2}\right|+\left|X\right|+\left|\text{core}(G)\right|
	\end{eqnarray*}
	Note that $S_{1}\cup S_{2}\cup X\subseteq\text{corona}(G)$ and for
	each vertex $v\in\text{corona}(G)-(S_{1}\cup S_{2}\cup X)$ there exists
	an odd cycle containing $v$ in $G[S_{1}\cup S_{2}\cup X\cup\{v\}]$. As the cycles obtained for vertices in $\textnormal{corona}(G)-(S_1 \u S_2 \u X)$ are vertex distinct, this implies
	\[
	k\ge\left|\text{corona}(G)-(S_{1}\cup S_{2}\cup X)\right|.
	\]
	Therefore
	\begin{align*}
		2\alpha(G)+k & \ge\left|\text{corona}(G)-(S_{1}\cup S_{2}\cup X)\right|+\left|S_{1}\cup S_{2}\cup X\right|+\left|\text{core}(G)\right|\\
		& =\left|\text{corona}(G)\right|+\left|\text{core}(G)\right|.
	\end{align*}
	As desired.
\end{proof}

A similar argument to the proof of \cref{mainlema1} shows the following.

\begin{remark}
	For every graph $G$ of order $n$ we have 
	\[
	n\le2\alpha(G)+k,
	\]
	\noindent where $k$ is the maximum cardinality of a set of
	vertex-distinct odd cycles of $G$. 
\end{remark}

By \cref{levit1}, we have $\left|\textnormal{corona}(G)\right|+\left|\textnormal{ker}(G)\right|\le\left|\textnormal{corona}(G)\right|+\left|\textnormal{core}(G)\right|$.
Thus $k$ in \cref{mainlema1} is clearly a lower bound for the
number of odd cycles in the graph $G$. Therefore, this verifies
\cref{babconj}. 

\begin{corollary}
	For every graph $G$, 
	\[
	\left|\textnormal{corona}(G)\right|+\left|\textnormal{core}(G)\right|\le2\alpha(G)+k,
	\]
	\noindent where $k$ is the number of odd cycles in $G$. 
\end{corollary}

Thus our next goal is to prove \cref{mainlema2} in order to verify \cref{asdasd1}.

\begin{theorem}
	[\label{19}\cite{berge2005some}] An independent set
	$S$ is maximum if and only if every independent set disjoint
	from $S$ can be matched into $S$. 
\end{theorem}

The following Lemma is a straightforward application of Hall's Theorem.

\begin{lemma}
	[\cite{larson2007note}\label{asdm1}]  Let $G$ be a graph and $I$ a
	critical independent set of $G$. Then there exists a maximum matching
	of $G$ that matches the vertices $N(S)$ into (a subset of) the
	vertices of $I$.
\end{lemma}

From \cref{19} and \cref{asdm1} we obtain \cref{lem: ind critico ind max}.

\begin{lemma}\label{lem: ind critico ind max}
	Let $I_c$ be a critical independent set and $I_M$ is a maximum independent set. If 
	\begin{align*}
		C=&I_c\cap I_M,&A=&I_c\setminus C,& B=N(A)\cap I_M,
	\end{align*}
	then
	$|A|=|B|$, and $(I_M\setminus B)\cup A$ is a maximum independent set. Moreover, $C$ is a critical independent set.
	
\end{lemma}

\begin{proof}
	We prove that $C$ is a critical set. The rest follows directly
	from \cref{19} and \cref{asdm1}. Note that 
	\[
	N(C)\subseteq N(I_{c})-B.
	\]
	Therefore,
	\begin{eqnarray*}
		\left|C\right|-\left|N(C)\right| & \ge & \left|C\right|-\left|N(I_{c})-B\right|\\
		& = & \left|C\right|-\left|N(I_{c})\right|+\left|B\right|\\
		& = & \left|C\right|+\left|B\right|-\left|N(I_{c})\right|\\
		& = & \left|I_{c}\right|-\left|N(I_{c})\right|.
	\end{eqnarray*}
	Thus, $C$ is a critical independent set.
\end{proof}

\begin{corollary}\label{cor: ind - B + A es ind}
	Let $I_c$ be a critical independent set and $I_M$ is a maximum independent set. If 
	\begin{align*}
		C=&I_c\cap I_M,&A=&I_c\setminus C,& B=N(A)\cap I_M,
	\end{align*}
	then $(I_M \cup A)\setminus B$ is a maximum independent set containing $I_c$. 
\end{corollary}

Observe that, by \cref{lem: ind critico ind max}, we have $\text{ker}(G)\subseteq I_{M}$
for every maximum independent set $I_{M}$. This yields
a short proof of \cref{levit1}.

\begin{lemma}
	[\label{9} \cite{levit2012critical}] For a graph $G$, if $I$ and $J$ are critical in $G$, then $I\u J$ and $I\ii J$ are critical as well.
\end{lemma}

\begin{lemma}\label{lem: 1 critico, varios maximos}
	Let $I_c$ be a critical independent set and let $I_1^M,\ldots, I_n^M$ be a family of maximum independent sets. 
	Let $C=I_c\cap( \bigcup  I_j^M)$, and let $A=I_c\setminus C$ 
	and $B=N(A)\cap (\bigcap I_j^M)$. Then $|B|=|A|$.
\end{lemma}

\begin{proof}
	We prove this by induction on $n$. For the base case of the proof, see \cref{Figura312}. For each $i=1,\dots,n$
	define $B_{i}=N(I_{c})\cap I_{1}^{M}$ and $A_{i}=I_{c}-I_{1}^{M}$.
	By \cref{asdm1}, note that $\left|B_{1}\cap B_{2}\right|\le\left|A_{1}\cap A_{2}\right|.$
	By \cref{lem: ind critico ind max}, 
	\[
	\left|A_{1}\right|=\left|B_{1}\right|,\,\left|A_{2}\right|=\left|B_{2}\right|
	\]
	and $I_{1}^{M}\cap I_{c},I_{2}^{M}\cap I_{c}$ are critical sets.
	Then, by \cref{9}, $I_{1}^{M}\cap I_{2}^{M}\cap I_{c}$ is a critical
	independent set and consequently $\left|B_{1}\cup B_{2}\right|\le\left|A_{1}\cup A_{2}\right|$,
	that is,
	\[
	\left|B_{1}\right|+\left|B_{2}\right|-\left|B_{1}\cap B_{2}\right|\le\left|A_{1}\right|+\left|A_{2}\right|-\left|A_{1}\cap A_{2}\right|.
	\]
	Finally, we get $\left|B_{1}\cap B_{2}\right|=\left|A_{1}\cap A_{2}\right|.$
	
	For the inductive step, observe that a similar argument shows that 
	\[
	I_{c}\cap\bigcap_{i=1}^{n}I_{i}^{M}
	\]
	is a critical independent set and therefore 
	\[
	\left|\bigcup_{i=1}^{n}B_{i}\right|\le\left|\bigcup_{i=1}^{n}A_{i}\right|.
	\]
	For the reverse inequality, we show that there is a matching from $\bigcup_{i=1}^{n}A_{i}$ into $\bigcup_{i=1}^{n}B_{i}$.
	Consequently, by the inclusion–exclusion principle,
	\[
	\sum_{k=1}^{n}(-1)^{k+1}\left|\bigcap_{\underset{\left|I\right|=k}{i\in I\subseteq\{1,\dots,n\}}}B_{i}\right|
	=
	\sum_{k=1}^{n}(-1)^{k+1}\left|\bigcap_{\underset{\left|I\right|=k}{i\in I\subseteq\{1,\dots,n\}}}A_{i}\right|.
	\]
	By the inductive hypothesis, the above reduces to $\left|A\right|=\left|B\right|$,
	as desired. Therefore, it remains only to show that there is
	a matching from $\bigcup_{i=1}^{n}A_{i}$ into $\bigcup_{i=1}^{n}B_{i}$.
	
	By \cref{19}, there is a matching from $A_{1}$ into $B_{1}$. If
	$A_{1}=\bigcup_{i=1}^{n}A_{i}$ we are done. Otherwise, without loss of generality,
	$I_{2}^{M}$ is a maximum independent set not containing $x$. Then,
	analogously, there is a matching from $A_{2}-A_{1}$ into $I_{2}^{M}$.
	Similarly, if $y\in\bigcup_{i=1}^{n}A_{i}-\left(A_{1}\cup A_{2}\right)$,
	then $I_{3}^{M}$ is a maximum independent set not containing $y$,
	and hence there is a matching from $A_{3}-(A_{1}\cup A_{2})$ into $I_{3}^{M}$.
	Continuing in this way, we obtain the desired matching.
\end{proof}

\begin{figure}[h!]
	
	\begin{center}

		\tikzset{every picture/.style={line width=0.75pt}} 
		
		\begin{tikzpicture}[x=0.75pt,y=0.75pt,yscale=-1,xscale=1]
			
			\draw   (87.42,161.92) .. controls (98.63,136.54) and (161.68,139.79) .. (228.26,169.18) .. controls (294.84,198.58) and (339.73,242.98) .. (328.53,268.36) .. controls (317.32,293.74) and (254.27,290.49) .. (187.69,261.1) .. controls (121.11,231.7) and (76.22,187.3) .. (87.42,161.92) -- cycle ;
			\draw   (236.47,268.14) .. controls (225.54,242.64) and (270.91,198.73) .. (337.8,170.06) .. controls (404.69,141.39) and (467.78,138.82) .. (478.71,164.32) .. controls (489.64,189.82) and (444.27,233.74) .. (377.38,262.41) .. controls (310.49,291.08) and (247.4,293.65) .. (236.47,268.14) -- cycle ;
			\draw   (245.28,156.97) .. controls (245.28,107.1) and (261.41,66.67) .. (281.31,66.67) .. controls (301.2,66.67) and (317.33,107.1) .. (317.33,156.97) .. controls (317.33,206.84) and (301.2,247.27) .. (281.31,247.27) .. controls (261.41,247.27) and (245.28,206.84) .. (245.28,156.97) -- cycle ;
			\draw  [fill={rgb, 255:red, 144; green, 19; blue, 254 }  ,fill opacity=0.1 ][dash pattern={on 4.5pt off 4.5pt}] (205.33,221.67) .. controls (216.33,212.67) and (238.33,239.67) .. (253.33,246.67) .. controls (268.33,253.67) and (296.33,243.67) .. (296.33,265.67) .. controls (296.33,287.67) and (236.67,277.33) .. (212.33,259.67) .. controls (188,242) and (194.33,230.67) .. (205.33,221.67) -- cycle ;
			\draw  [dash pattern={on 4.5pt off 4.5pt}] (360.88,221.67) .. controls (349.09,212.67) and (325.51,239.67) .. (309.43,246.67) .. controls (293.35,253.67) and (263.33,243.67) .. (263.33,265.67) .. controls (263.33,287.67) and (327.29,277.33) .. (353.38,259.67) .. controls (379.46,242) and (372.67,230.67) .. (360.88,221.67) -- cycle ;
			\draw  [fill={rgb, 255:red, 144; green, 19; blue, 254 }  ,fill opacity=0.1 ] (281.31,66.67) .. controls (312.42,71.93) and (320.2,131.4) .. (316.33,181.67) .. controls (312.47,231.93) and (302.2,213.4) .. (281.33,197.67) .. controls (260.47,181.93) and (255.33,183.67) .. (247.33,178.67) .. controls (239.33,173.67) and (250.2,61.4) .. (281.31,66.67) -- cycle ;
			
			\draw (84,124.4) node [anchor=north west][inner sep=0.75pt]  [font=\large]  {$I_{1}^{M}$};
			\draw (462,127.4) node [anchor=north west][inner sep=0.75pt]  [font=\large]  {$I_{2}^{M}$};
			\draw (276,44.4) node [anchor=north west][inner sep=0.75pt]  [font=\large]  {$I_{c}$};
			\draw (210,233.4) node [anchor=north west][inner sep=0.75pt]    {$B_{1}$};
			\draw (339,233.4) node [anchor=north west][inner sep=0.75pt]    {$B_{2}$};
			\draw (279,137.4) node [anchor=north west][inner sep=0.75pt]    {$A_{1}$};

		\end{tikzpicture}

	\end{center}
	\caption{Illustration of the proof of \cref{mainlema2}.}
	\label{Figura312}
\end{figure}
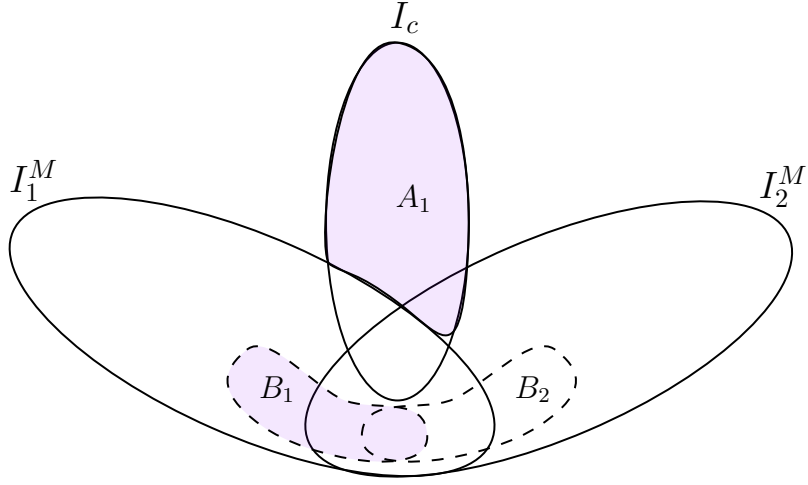

\begin{theorem}[\cite{butenko2007using}\label{asdnjh1}]
	If $I_{c}$ is a critical independent set in a graph $G$ then there
	is a maximum independent set $I$ in $G$ such that $I_{c}\s I$. 
\end{theorem}

\begin{theorem}\label{mainlema2}
	For every graph $G$, 
	\[
	\left|\textnormal{nucleus}(G)\right|+\left|\textnormal{diadem}(G)\right|\le2\alpha(G).
	\]
\end{theorem}

\begin{proof}
	Let $I_1,I_2,\ldots,I_m$ the family of critical maximum independent sets of $G$. By \cref{asdnjh1}, there exists a maximum independent set $I_1^M$ containing $I_1$, and for $j\geq 2$ let 
	$I_j^M=(I_1^M\cup I_j)\setminus (N(I_j)\cap I_1^M)$. By Lemma \ref{lem: ind critico ind max},
	$I_j^M$ is a maximum independent set containing $I_j$.
	
	For $2 \leq j \leq m$ we let 
	\begin{align*}
		\mathcal{U}_j=&\bigcup_{i=1}^{j-1} I^M_i\\
		\mathcal{L}_j=&\bigcap_{i=1}^{j-1} I^M_i\\
		\mathcal{A}_j=&I_j\setminus \mathcal{U}_{j}\\
		\mathcal{B}_j=&N(\mathcal{A}_j)\cap \mathcal{L}_j.
	\end{align*} 
	By Lemma \ref{lem: 1 critico, varios maximos}, we get that $|\mathcal{B}_j|=|\mathcal{A}_j|$.
	Further, since $\mathcal{B}_j\cap I_j=\emptyset$, we get $\mathcal{B}_j\subset \mathcal{L}_j\setminus \mathcal{L}_{j+1}$. 
	This implies that $\mathcal{B}_j\cap \mathcal{B}_\ell =\emptyset$ if $1\leq j<\ell\leq m$. Further, notice that $\textnormal{nucleus} (G)\s I_j$ implies $\mathcal{B}_j\cap \textnormal{nucleus}(G)=\emptyset$. Therefore $\bigcup_{i=2}^{m}\mathcal{B}_i\subset \mathcal{L}_2\setminus \textnormal{nucleus}(G)$. But $\mathcal{L}_2=I^M_1$.
	
	Notice that 
	\[
	\operatorname{diadem}(G)=\bigcup_{j=1}^m I_j \subseteq I_1^M\cup \bigcup_{j=2}^m \mathcal{A}_j.
	\]
	Thus,
	\begin{align*}
		|\operatorname{diadem}(G)|=&\left|\bigcup_{j=1}^m I_j\right|\\
		\leq &|I_1^M|+ \sum_{j=2}^m |\mathcal{A}_j|\\
		=&|I_1^M|+ \sum_{j=2}^m |\mathcal{B}_j|\\
		\leq &|I_1^M|+ |I_1^M|-|\textnormal{nucleus}(G)|\\
	\end{align*}
	Therefore,
	\[
	|\operatorname{diadem}(G)|\leq 2\alpha (G) - |\textnormal{nucleus}(G)|.
	\]
	As desired.
\end{proof}

Thus, since $\text{ker}(G)\s\text{nucleus}(G)$, \cref{asdasd1} follows from \cref{mainlema2}.

\begin{theorem}[\cite{levit2014set}\label{anterior}]
	For every graph $G$, 
	\[
	\left|\textnormal{corona}(G)\right|+\left|\textnormal{core}(G)\right|\ge2\alpha(G).
	\]
\end{theorem}

Hence, from \cref{anterior}, \cref{mainlema1}, and \cref{mainlema2} we obtain the following
result.

\begin{theorem}\label{mainthm}
	For every graph $G$, 
	\begin{eqnarray*}
		\left|\textnormal{nucleus}(G)\right|+\left|\textnormal{diadem}(G)\right| & \le & 2\alpha(G)\\
		& \le & \left|\textnormal{corona}(G)\right|+\left|\textnormal{core}(G)\right|\\
		& \le & 2\alpha(G)+k.
	\end{eqnarray*}
	\noindent where $k$ is the maximum cardinality of a set of
	vertex-distinct odd cycles of $G$. 
\end{theorem}

It is possible to find graphs that satisfy the inequalities in \cref{mainthm}
with strict inequalities. In a complete graph of odd order $n\ge5$,
the empty set is the only critical independent set and $\left|\text{nucleus}(G)\right|+\left|\text{diadem}(G)\right|=0$.
However, $2\alpha(G)=2\left\lfloor \frac{n}{2}\right\rfloor =n-1<n$,
while $\left|\text{corona}(G)\right|+\left|\text{core}(G)\right|=n$,
$2\alpha(G)+k=n-1+k$ and $k\ge2$.

\section{Open Problems}\label{sss5}

It is an easy exercise to verify that bipartite graphs satisfy
all the inequalities in \cref{mainthm} with equality. This motivates the following.

\begin{problem}
	Characterize the graphs for which $\left|\textnormal{nucleus}(G)\right|+\left|\textnormal{diadem}(G)\right|=2\alpha(G)+k$,
	where $k$ is the maximum cardinality of a set of vertex-distinct
	odd cycles of $G$. 
\end{problem}

\begin{problem}
	Characterize the graphs for which $\left|\textnormal{corona}(G)\right|+\left|\textnormal{core}(G)\right|=2\alpha(G)+k$,
	where $k$ is the maximum cardinality of a set of vertex-distinct
	odd cycles of $G$. 
\end{problem}

\begin{problem}
	Is it possible to decide whether a graph satisfies $\left|\textnormal{nucleus}(G)\right|+\left|\textnormal{diadem}(G)\right|=2\alpha(G)+k$
	in polynomial time?
\end{problem} 

\begin{problem}
	Find a lower bound for $\left|\textnormal{nucleus}(G)\right|+\left|\textnormal{diadem}(G)\right|$
	in terms of $k$. 
\end{problem}

\begin{problem}
	Characterize the graphs that satisfy all the
	inequalities in \cref{mainthm} with strict inequality.
\end{problem}

\begin{problem}
	Given a graph of odd order $n\ge5$, find
	a lower bound for the number of edges such that the graph satisfies
	all the inequalities in \cref{mainthm} with strict inequality.
\end{problem}

\section*{Acknowledgments}

This work was partially supported by Universidad Nacional de San Luis, grants PROICO 03-0723 and PROIPRO 03-2923, MATH AmSud, grant 22-MATH-02, Consejo Nacional de Investigaciones
Cient\'ificas y T\'ecnicas grant PIP 11220220100068CO and Agencia I+D+I grants PICT 2020-00549 and PICT 2020-04064.

\section*{Declaration of generative AI and AI-assisted technologies in the writing process}
During the preparation of this work the authors used ChatGPT-3.5 in order to improve the grammar of several paragraphs of the text. After using this service, the authors reviewed and edited the content as needed and take full responsibility for the content of the publication.

\section*{Data availability}

Data sharing not applicable to this article as no datasets were generated or analyzed during the current study.

\section*{Declarations}

\noindent\textbf{Conflict of interest} \ The authors declare that they have no conflict of interest.

\bibliographystyle{apalike}

\bibliography{TAGcitasV2025}

\end{document}